\documentclass{amsart}

\usepackage{amsfonts,amsmath,amssymb}
\usepackage{latexsym,amscd,graphicx} 

\newcommand{\bg}{{\overline{g}}}
\newcommand{\bu}{{\overline{u}}}

\newcommand{\bz}{{\overline{z}}}

\newcommand{\be}{{\overline{e}}}
\newcommand{\bG}{{\overline{G}}}

\newcommand{\ve}{\varepsilon}
\newcommand{\vp}{\varphi}
\newcommand{\ld}{\ldots}
\newcommand{\cg}{C^{\gr}}

\DeclareMathOperator*{\ot}{\otimes}

\newcommand{\beq}{\begin{equation}}
\newcommand{\eeq}{\end{equation}}
\newcommand{\beas}{\begin{eqnarray*}}
\newcommand{\eeas}{\end{eqnarray*}}
\newcommand{\id}{\mathrm{id}}

\newcommand{\cM}{\mathcal{M}}

\newcommand{\R}{\mathbb{R}}
\newcommand{\C}{\mathbb{C}}
\renewcommand{\H}{\mathbb{H}}
\newcommand{\zz}{\mathbb{Z}_2\times\mathbb{Z}_2 }
\newcommand{\z}{\mathbb{Z}_2}
\newcommand{\NN}{\mathbb{N}}
\newcommand{\ZZ}{\mathbb{Z}}

\newcommand{\FF}{\mathbb{F}}

\DeclareMathOperator{\End}{\mathrm{End}}

\DeclareMathOperator{\alg}{\mathrm{alg}}

\DeclareMathOperator{\Ker}{\mathrm{Ker}\,}

\DeclareMathOperator{\supp}{\mathrm{Supp}\,}
\DeclareMathOperator{\Supp}{\mathrm{Supp}}

\newtheorem{theorem}{Theorem}[section]

\newtheorem{lemma}[theorem]{Lemma}

\newtheorem{proposition}[theorem]{Proposition}
\newtheorem{definition}[theorem]{Definition}

\newtheorem{remark}[theorem]{Remark}

\newcommand{\gr}{\mathrm{gr}}

\begin{document}

\title{Simple Graded Division Algebras over the Field of Real Numbers}

\author[Bahturin]{Yuri Bahturin}
\address{Department of Mathematics and Statistics, Memorial
University of Newfoundland, St. John's, NL, A1C5S7, Canada}
\email{bahturin@mun.ca}

\author[Zaicev]{Mikhail Zaicev}
\address{Faculty of Mathematics and Mechanics, Moscow University, Russia}
\email{mzaicev@mail.ru}

\thanks{The first author acknowledges support by by NSERC grant \# 227060-09. The second author acknowledges support by }

\begin{abstract}
We classify, up to equivalence, all finite-dimensional  simple graded division algebras over the field of real numbers. The grading group is any finite abelian group.
\end{abstract}

\maketitle

\section{Introduction} A unital algebra $R$ over a field $F$ graded by a group $G$ is called \textit{graded division} if every nonzero homogeneous element is invertible. Clearly, each such algebra is graded simple, that is, $R$ has non nonzero graded ideals. In the classification of gradings on simple finite-dimensional algebras one is interested in graded division algebras which are simple at the same time. Indeed, according to graded analogues of Schur's Lemma and Density Theorem (see, for example, \cite[Chapter 1]{EK}) any such algebra is isomorphic to the algebra $\End_DV$ of endomorphisms of a finite-dimensional graded (right) vector space over a graded division algebra $D$. Since $R$ is simple, it is obvious that $D$ must be simple, as well. 

In the case where the field $F$ is algebraically closed and the group $G$ is finite abelian, all such graded division algebras have been described in \cite{BSZ} and \cite{BZ02}. For full account see \cite[Chapter 1]{EK}, where the authors treat also the case of Artinian algebras. In \cite{BBZ} (see also \cite{BZ10}, for a particular case) the authors treat the case of primitive algebras with minimal one-sided ideals. If such algebras are locally finite, the graded division algebras arising are finite-dimensional and so the description provided in the case of finite-dimensional algebras works in this situation, as well.

The main results of this paper are Theorem \ref{tgdar} and Theorem \ref{tM}. In Theorem \ref{tgdar} we list all equivalence classes of division gradings on simple finite-dimensional real associative algebras. In Theorem \ref{tM} we apply these results to the classification of all abelian group gradings on such algebras. A special feature of the real case is the usage of Clifford algebra, which provide a natural approach to the study of gradings in the case of real algebras.

\section{Preliminaries}\label{sPre}

A vector decomposition $\Gamma: V=\bigoplus_{g\in G}V_g$ is called a grading of a vector space $V$ over a field $F$ by a set $G$. The subset $S$ of all $s\in G$ such that $V_s\neq \{ 0\}$ is called the \textit{support} of $\Gamma$ and is denoted by $\Supp\Gamma$ (also as $\Supp V$, if $S$ is endowed just by one grading). If $\Gamma': V'=\bigoplus_{g'\in G'}V'_{g'}$ is a grading of another space then a homomorphism of gradings $\vp:\Gamma\to\Gamma'$ is a linear map $f:V\to V'$ such that for each $g\in G$ there exists (unique) $g'\in G'$ such that $\vp(V_g)\subset V_{g'}$. If $\vp$ has an inverse as homomorphism of grading then we say that $\vp:\Gamma\to\Gamma'$  is an \textit{equivalence} of gradings $\Gamma$ and $\Gamma'$ (or graded vector spaces $V$ and $V'$).

A grading $\Gamma: R=\bigoplus_{g\in G}R_g$ of an algebra $R$ over a field $F$ is an \textit{algebra} grading if for any $s_1,s_2\in\Supp\Gamma$ such that $R_{s_1}R_{s_2}\ne \{0\}$ there is $s_3\in G$ such that $R_{s_1}R_{s_2}\subset R_{s_3}$. Two algebra gradings $\Gamma: R=\bigoplus_{g\in G}R_g$ and $\Gamma': R'=\bigoplus_{g\in G'}R'_{g'}$ of algebras over a field $\FF$ are called \textit{equivalent} if there exist an algebra isomorphism $\vp:R\to R'$, which is an equivalence of vector space gradings. In this case there is a bijection $\alpha:\Supp\Gamma\to\Supp\Gamma'$ such that $\vp(R_g)=R'_{\alpha(g)}$.

If $G$ is a group then a grading $\Gamma: R=\bigoplus_{g\in G}R_g$ of an algebra $R$ over a field $F$ is called a \textit{group grading} if for any $g,h\in G$, we have $R_gR_h\subset R_{gh}$. Normally, it is assumed that the grading group $G$ is generated by $\Supp\Gamma$. If $\vp:\Gamma\to Gamma':R\to R'$ is an equivalence of gradings of algebras $R$ and $R'$ by groups $G$ and $G'$ and the accompanying bijection $\alpha:\Supp\Gamma\to\Supp\Gamma'$ comes from an isomorphism of groups $\alpha:G\to G'$ then we call $\vp$ a \textit{weak isomorphism} and say that $\Gamma$ and $\Gamma'$ (also $R$ and $R'$) are \textit{weakly isomorphic}. Finally, if $G=G'$ and $\alpha=\id_G$ then $\Gamma$ and $\Gamma'$ are called \textit{ isomorphic}.

Note that if, say, $\Gamma$ is a \textit{strong grading}, that is, $R_gR_h=R_{gh}$ then $\Gamma$ and $\Gamma'$ are equivalent if and only if they are weakly isomorphic. 

If $\Gamma: R=\bigoplus_{g\in G}R_g$ and $\Gamma': R=\bigoplus_{g'\in G'}R_{g'}^\prime$ are two gradings of the same algebra labeled by the sets $S$ and $S'$ then we say that $\Gamma$ is the \textit{refinement} of $\Gamma'$ if for any $s\in S$ there is $s'\in S'$ such that $R_s\subset R_{s'}$. We also say that $\Gamma'$ is the \textit{coarsening} of $\Gamma$. The refinement $\Gamma$ is proper if for at least one $s$ the containment $R_s\subset R_{s'}$ is proper. A grading which does not admit proper refinements is called \textit{fine}. Assume $\Gamma,\Gamma'$ are group gradings, so that $G'=G/T$. If $R^\prime_{\bg}=\bigoplus_{g\in\bg}R_g$, for all $\bg\in G'$, then $\Gamma'$ is a coarsening of $\Gamma$ called \textit{factor-grading}. In the case of complex gradings all division gradings are fine, while in the case of real numbers this is no more true.
      
Finally, if $\Gamma: R=\bigoplus_{g\in G}R_g$, the group $U(\Gamma)$ generated by $\Supp(\Gamma)$ subject to defining relations $s_1s_2=s_3$ each time when $\{ 0\}\neq R_{s_1}R_{s_2}\subset R_{s_3}$, is called the\textit{ universal group} of $\Gamma$.

\subsection{Basic properties of division gradings}\label{ssBP}
We start with fixing few well-known useful properties (see in the textbook \cite[Chapter 2]{EK}).
\begin{lemma}\label{lsf} Let $\Gamma: R=\bigoplus_{g\in G} R_g$ be a grading by a group $G$ on an (associative) algebra $R$ over a field $F$. Then the following hold.
\begin{enumerate}
\item $\Supp(\Gamma)$ is a subgroup of $G$ isomorphic to the universal group $U(\Gamma)$
\item $\Gamma$ is a division grading if and only if the identity component $R_e$ of $\Gamma$ is a division algebra over $F$
\item Given $g\in G$ and a nonzero $a\in R_g$, we have $R_g=aR_e$
\item $\dim R_g=\dim R_e$ and $\dim R=|\Supp\Gamma|\dim R_e$.\hfill$\square$
\end{enumerate}
\end{lemma}

Since our base field is $\R$, it follows that $R_e$ is one of $\R$, $\C$, or $\H$, the last two being the field of complex numbers or the division algebra of quaternions,  respectively. Both $\C $ and $ \H $ can be endowed by nontrivial real gradings by $\z $ in the case of $\C $ and $ \H $ and $\zz $ in the case of $ \H $. These are, for example, $ \C=\langle 1\rangle\oplus \langle i\rangle$ and $ \H=\langle 1, i \rangle\oplus \langle j, k\rangle$, in the case of $\z $ and $ \H=\langle 1 \rangle\oplus \langle i\rangle \oplus \langle j\rangle \oplus \langle k\rangle$ in the case of $\zz $ . We will denote the above graded algebras by $\C^{ (2)}$, $\H^{ (2)}$ and $\H^{ (4)}$.

As mentoned above, the support of the division grading is a subgroup in the grading group. This makes it natural to always assume that the support of $R$ equals the whole of $G$. Clearly, every homogeneous component $R_g$ of $R$, $g\in G$,  has the same dimension as $R_e$ and hence $\dim R=|G|\dim R_e$.  Thus, when we speak about gradings on finite-dimensional division graded algebras, we may assume that the grading group $G$ is finite.

One notational remark. Given an element $g$ of order $n$ in a group $G$, we denote by $(g)_n$ the cyclic subgroup generated by $g$. Given vectors $v_1,\ld,v_m$ in a real vector space $V$, we denote by $\langle v_1,\ld,v_m\rangle$ the linear span of $v_1,\ld,v_m$, with coefficients in $\R$. To avoid confusion with number $1\in\R$, we will denote the identity element of a graded division algebra $R$ by $I$. 

\subsection{Sylvester/Pauli gradings}\label{ssp}
 
 One of the important cases of real simple algebras are the algebras $M_n(\C)$ of $n\times n$ complex matrices. This real algebra is at the same time an algebra over $\C$. If $\Gamma:R=\bigoplus_{g\in G}R_g$ is a grading of $R$ over $\C$ then it is also a grading of $R$ over $\R$. We denote this latter grading by $\Gamma_\R$. If $\Gamma$ is a division grading then, clearly, $\Gamma_\R$ is a division grading. We will call such gradings Pauli gradings. Given another complex division grading $\Gamma'$ of $R$ by $G$, if $\Gamma_\R$ and $\Gamma'_\R$ are weakly isomorphic then there exist a group automorphism $\alpha:G\to G$ and automorphism $\vp$ of $R$, as a real algebra, such that  $\vp(R_g)=R'_{\alpha(g)}$.  We have $R_e=R'_e\cong\C$ and so $\vp$ is an isomorphism of $\C$, that is, either the restriction of the identity map  $\id_R$ or of the complex conjugation $\iota:R\to R: A\mapsto\overline{A}$. In both cases,  either $(\alpha,\vp)$ or $(\alpha,\vp\circ\iota)$ perform a weak equivalence of $\Gamma$ and $\Gamma'$. 
 
 The result is that two Pauli gradings are weakly isomorphic if and only if they come from weakly isomorphic complex division gradings. Such classification has been performed in \cite{BZ03}.

In that classification we have used  complex \textit{generalized Pauli} or \textit{Sylvester}  matrices, which define a standard division $\ZZ_n\times\ZZ_n$-grading on the matrix algebras $M_n(\C)$, over the field $\C$ of complex numbers. These are products of two generating matrices

\begin{equation}\label{ePauli}
X_a=\left(\begin{array}{ccccc}1&0&0&\cdots&0\\0&\ve&0&\cdots&0\\0&0&\ve^2&\cdots&0\\\cdots&\cdots&\cdots&\cdots&\cdots\\
0&0&0&\cdots&\ve^{n-1}\end{array}\right),\;X_b=\left(\begin{array}{ccccc}0&1&0&\cdots&0\\0&0&1&\cdots&0\\0&0&0&\cdots&0\\\cdots&\cdots&\cdots&\cdots&\cdots\\
1&0&0&\cdots&0\end{array}\right)
\end{equation}
Here $\ve$ is a primitive $n$-th root of 1 in $\C$. Note that Sylvester called $X_\alpha$ a \textit{clock} matrix and $X_\beta$ a \textit{shift} matrix. If $\ZZ_n\times\ZZ_n=(\alpha)_n\times(\beta)_n$ then the division grading in question is done by $n^2$ one-dimensional subspaces $R_{\alpha^k \beta^\ell}=\langle X_\alpha^kX_\beta^\ell\rangle_\C$, where $1\le k,\ell\le n$. According to \cite{BZ03}, every complex division grading on $M_n(\C)$ is weakly isomorphic to the (complex) tensor product $M_{k_1}(\C)\ot\cdots\ot M_{k_s}(\C)$ of the gradings just described. The choice of primitive roots $\varepsilon$ does not change the \textit{equivalence class} of these gradings. If $G=\ZZ_{k_1}^2\times\cdots\times \ZZ_{k_1}^2$ then we denote such graded algebra by $M(G)$.

One can describe Pauli gradings using skew-symmetric bicharacters on the grading group $G$. If $R=M_n(\C)$ is given a complex division grading with support $G$ then each $R_g$ is spanned by an invertible matrix $X_g$ such that $X_g^n=I$ and there is a function $\beta:G\times G\to \C^\times$ such that $X_gX_h=\beta(g,h)X_hX_g$. This function is a \textit{skew-symmetric bicharacter}, that is $\beta(g,h)\beta(h,g)=1$, for all $g,h\in G$ and $\beta(gh,k)=\beta(g,k)\beta(h,k)$, for all $g,h,k\in G$. The bicharacter $\beta$ is \textit{non-singular}, that is, for any $g\in G$ there is $h\in G$ such that $\beta(g,h)\neq 1$. Two division gradings $\Gamma$ and $\Gamma'$ with supports $G$ and $G'$ and skew-symmetric bicharacters $\beta$ and $\beta'$ are isomorphic if and only if there exists an isomorphism $\alpha:G\to G'$  such that $\beta'(\alpha(g),\alpha(h))=\beta(g,h)$, for all $g,h\in G$. For the equivalence, it is only necessary and sufficient to have the isomorphism $\alpha: G\to G'$. One more remark is that a division grading by a finite abelian group $G$ on $M_n(\C)$ exists (and unique up to equivalence) if and only if there is an abelian group $H$ of order $n$ such that $G\cong H\times H$ and the bicharacter $\beta$ is non-singular. For these facts see \cite[Theorem 2.38]{EK}. We denote the graded division algebras arising  in this way by $M(G,\beta)$.

Sylvester matrices of order 2 are especially important for the gradings on the real simple algebras. They are as follows:
\begin{equation}\label{ePauli2}
A=\left(\begin{array}{cc}1&0\\0&-1\end{array}\right)\!,\; B=\left(\begin{array}{cc}0&1\\1&0\end{array}\right)\!,\; C=\left(\begin{array}{cc}0&1\\-1&0\end{array}\right).
\end{equation}
Slightly different matrices are called \textit{Pauli matrices}:
\begin{equation}\label{ePauli3}
\sigma_x=\left(\begin{array}{cc}0&1\\1&0\end{array}\right)\!,\; \sigma_y=\left(\begin{array}{cc}0&-i\\i&0\end{array}\right)\!,\; \sigma_z=\left(\begin{array}{cc}1&0\\0&-1\end{array}\right).
\end{equation}
We have $A^2=I$, $B^2=I$, $C^2=-I$, $AB=C=-BA$, $BC=A=-AB$, $CA=B=-AC$. If $G=(\alpha)_2\times(\beta)_2\cong\zz$ and $\gamma=\alpha\beta$ then we have an important fine division grading of $R=M_2(\R)$, given by
$$
\Gamma: R_e=\langle I\rangle,\:R_\alpha=\langle A\rangle,\:R_\beta=\langle B\rangle,\:R_\gamma=\langle C\rangle.
$$
This is also Clifford grading $C^{\gr}(0,2)$, as can be seen from the next subsection. We denote this graded algebra by $ M_2^{ ( 4 ) } $. 

\subsection{Homogeneous Clifford gradings}\label{ssc}
Another source of gradings, which is essential for the case of algebras over the field $\R$ of real numbers, is as follows. Let $(V,q)$ be a real finite-dimensional vector space endowed with non-singular quadratic form $q:V\to \R$. Let $T(V)$ be the tensor algebra of $V$ and $J_q$ the two-sided ideal generated by all elements $x\ot x+q(x)1$, where $x\in V$. Then $C(V,q)=T(V)/J_q$ is called the \textit{Clifford algebra} of $(V,q)$. It is well-known that if $q$ is non-singular with positive inertia index $p$ and negative inertia index $m$ and $n=\dim V=2k$ is even then  $C(V,q)$ is central simple, that is, either $M_{2^k}(\R)$ or  $M_{2^{k-1}}(\H)$. More precisely, if $p-m\equiv 0$ or $6 \mod 8$ then $C(V,q)\cong M_{2^k}(\R)$. If $p-m\equiv 2$ or $4 \mod 8$ then $C(V,q)\cong M_{2^{k-1}}(\H)$.   Also, if $n=2k+1$ and $q$ has positive inertia index $p$ and negative $m$ with $p-m\equiv 1 \mod 4$ then $C(V,q)\cong M_{2^k}(\C)$.

Now let $G$ be an elementary abelian 2-group and $g_1,\ld,g_n\in G$. Choose an orthogonal basis $e_1,\ld,e_n$ in $V$ and set $\deg e_1=g_1,\ld \deg e_n=g_n$. Then a natural grading appears on the tensor algebra $T(V)$ if one sets $\deg(e_{i_1}\ot\cdots\ot e_{i_k})=g_{i_1}\cdots g_{i_k}$. If $x\in V$ has the form of $x=\sum_{k=1}^n\alpha_ke_k$ then the generators of $J_q$ can be written as 
$$
x\ot x+q(x).1=\sum_{k=1}^n \alpha_1^2(e_k\ot e_k+q(e_k)1)+\sum_{1\le k<\ell\le n}^n\alpha_k\alpha_\ell(e_ke_\ell+e_\ell e_k).
$$
Thus $J_q$ is generated by $G$-homogeneous elements $e_k\ot e_k+q(e_k)1$, $1\le k\le n$ and  $e_ke_\ell+e_\ell e_k$ where $1\le k<\ell\le n$. Hence $C(V,q)=T(V)/J_q$ acquires a (canonical) factor-grading by $G$, given on the canonical basis by $\deg e_{i_1}\cdots e_{i_k}=g_{i_1}\cdots g_{i_k}$.

Since $G$ is an elementary abelian 2-group, one can view $G$ as a vector space over the field $\z$. If all $g_1,\ldots, g_n$ are linearly independent (over $\z$) then the grading has all components one-dimensional (over $\R$) and generated by invertible elements. In this case we have a \textit{fine} graded division algebra. Any other choice of linearly independent elements in $G$ leads to weakly isomorphic algebras. If $p$ is the positive and $m$ the negative inertia indices of $q$ then we denote the respective graded division algebra by $\cg(p,m)$.  As a result, we have the following.

\begin{proposition}\label{pc}
Each simple algebra $R=M_{2^k}(D)$, $D=\R,\C,\H$ acquires a fine division grading if one identifies $R$ with an appropriate $\cg(p,m)$.   
\end{proposition}

\subsection{Non-homogeneous Clifford gradings}\label{ssogc} Clifford algebras $C(p,m)$ with $p-m=1 \mod 4$ can be endowed with a $\ZZ_4\times\z^{p+q-1}$ division grading. Since we will need these gradings only for $M_2(\C)$ and $M_4(\C)$, we define them only in these cases. In the case of $R=C(2,1)$, we choose $G=(\alpha)_4\times (\beta)_2$, set  $R_{\alpha}= \langle e_1+e_2e_3\rangle$ and $R_\beta=\langle e_2\rangle$. Then, since $R_gR_h=R_{gh}$, for all $g,h\in G$, we obtain
$$R_{\alpha^2}=\langle e_1e_2e_3\rangle,\;R_{\alpha^3}= \langle e_1-e_2e_3\rangle.
$$
Also 
$$R_{\alpha\beta}=\langle e_3+e_1e_2\rangle,\; 
R_{\alpha^2\beta}=\langle e_1e_3\rangle,\;R_{\alpha^3\beta}=\langle e_3-e_1e_2\rangle.
$$
If we use one of the standard identifications (called \textit{Clifford maps}, see \cite[121]{G}) of $C(2,1)$ with $M_2(\C)$, given by $e_1\mapsto iB$,  $e_2\mapsto C$, $e_3\mapsto A$, then we will find the homogeneous components as 
\begin{eqnarray}\label{eq28}
R_e&=&\langle I\rangle,\;R_\alpha=\langle \omega A\rangle,\;R_{\alpha^2}=\langle iI\rangle,\;R_{\alpha^3}= \langle i\omega A\rangle\\
R_\beta&=&\langle C\rangle,\;R_{\alpha\beta}=\langle \omega B\rangle,\; 
R_{\alpha^2\beta}=\langle iC\rangle,\;R_{\alpha^3\beta}=\langle i\omega B\rangle.\nonumber
\end{eqnarray}
Here for brevity, instead of $1+i$, we have used $\omega=\frac{1}{\sqrt{2}}(1+i)$, the 8th root of 1 such that $\omega^2=i$. We denote this grading by $M_2^{(8)}$.

A $\ZZ_4$-division grading of $M_2(\C)$, which is not fine, appears as the coarsening of $M_2^{(8)}$ by means of identifying $\beta$ and $e$. 
\begin{equation}\label{eq28c}
R_\be=\langle I,C\rangle,\;R_{\overline{\alpha}}=\langle \omega A, \omega B\rangle,\;R_{\overline{\alpha}^2}=\langle iI,iC\rangle,\;R_{\overline{\alpha}^3}= \langle i\omega A,i\omega B\rangle.
\end{equation}

In terms of Clifford algebra $R=C(2,1)$, the same grading will be obtained if we set  $R_e=\langle I,e_2\rangle$, $R_{\alpha}= \langle e_1+e_2e_3, e_1e_2+e_3\rangle$ and $R_{\alpha^2}= \langle e_1e_2e_3, e_1e_3\rangle$ and $R_{\alpha^3}= \langle e_1-e_2e_3, e_1e_2-e_3\rangle$.

In this case, the identity component of the grading is isomorphic to $\C$.  We denote this grading by $M_2(\C,\ZZ_4)$

A $\ZZ_4\times\z^3$-grading on $C(3,2)\cong M_4(\C)$ appears if one select an element of order 8 in this algebra in the form $e_1+e_2e_3e_4e_5$ and write $G=(\alpha)_4\times(\beta)_2\times(\gamma)_2\times(\delta)_2$. We proceed as follows:
$$
R_\alpha=\langle e_1+e_2e_3e_4e_5\rangle,\;R_\beta=\langle e_2\rangle,\;R_\gamma=\langle e_3\rangle,\;R_\delta=\langle e_4\rangle.
$$
Inside this algebra, we find a graded subalgebra $S$ generated by $e_2,e_3,e_4$, which is isomorphic to $\cg(2,1)\cong M_2^{(8)}$. Furthermore, $R_{\alpha^2}=\langle e_1e_2e_3e_4e_5\rangle,\;R_{\alpha^3}=\langle e_1-e_2e_3e_4e_5\rangle$. An element $w\in S$ does not depend on $e_1,e_5$. At the same time, $(e_1+e_2e_3e_4e_5)
w$, for any monomial $w$ in $S$, is the sum of two monomials, one depending on $e_1$ but not on $e_5$, the other on $e_5$ but not on $e_1$. Also, $e_1e_2e_3e_4e_5w$ depends on both $e_1$ and $e_5$. This shows that $R$ is the direct sum of subspaces of the form $R_{\alpha^k\sigma}$, where $\sigma\in (\beta,\gamma,\delta)$.  As a result, we have a grading of $C(3,2)\cong M_4(\C)$ by $\ZZ_4\times\z^3$.

 Again, in terms of matrices, this grading can be rewritten as follows. We write $M_4(\C)=M_2(\C)\ot M_2(\R)$. Then a Clifford map can be given by mapping (see \cite[p.132]{G})
$$
e_1\mapsto (iI)\ot B,\:e_2\mapsto (iB)\ot A,\:e_3\mapsto (iA)\ot A,\:e_4\mapsto (iC)\ot A,\:e_5\mapsto (iI)\ot C,\:
$$
Then
$$
R_\alpha=\langle (\omega I)\ot B\rangle,\;R_\beta=\langle (iB)\ot A\rangle,\;R_\gamma=\langle (iA)\ot A\rangle,\;R_\delta=\langle (iC)\ot A\rangle.
$$ 
We will denote this grading by $M_4^{(32)}$. Again, a coarsening of this grading can be obtained if we identify $e$ with $\gamma$. Then the generating subspaces will be as follows.
$$
R_\be=\langle I\ot I, (iA)\ot A\rangle,\;R_{\overline{\alpha}}=\langle (\omega I)\ot B,(i\omega  A)\ot C\rangle,
$$
$$
R_{\overline{\beta}}=\langle (iB)\ot A,C\ot I\rangle,\;R_{\overline{\delta}}=\langle (iC)\ot A, B\ot I \rangle.
$$
so that this is a grading of $M_4(\C)$ by $\ZZ_4\times\z^2\cong G/(\gamma)_2$. We denote this grading by $M_4(\C,\ZZ_4\times\z^2)$.

In terms of $R=C(3,2)$, we can write the generating subspaces of same grading as $R_{\overline{\alpha}}=\langle e_1+e_2e_3e_4e_5,\,e_1e_3-e_2e_4e_5\rangle$, $R_{\overline{\beta}}=\langle e_2,e_2e_3,\rangle$, $R_{\overline{\delta}}=\langle e_4,e_3e_4\rangle$.

We will see, however, that $M_4^{(32)}$ is weakly isomorphic to the tensor product of $M_2^{(8)}$ and $M_2^{(4)}$.

An appropriate tool for establishing the equivalence or weak isomorphism of gradings is provided in the next subsection.

\subsection{Graded division algebras in terms of generators and defining relations}\label{ssgdr}
Let $A_m=F\langle x_1,\ld,x_m\rangle$ be a free associative unital algebra with free generators $x_1,\ld,x_m$. Let $r_1=r_1(x_1,\ld,x_m),\ld,r_\ell=r_\ell(x_1,\ld,x_m)$ be elements of $A_m$ and $J$ the two-sided ideal of $A_m$ generated by $r_1,\ld,r_\ell$. Then we say that $R=A_m/J$ is defined in terms of generators $x_1,\ld,x_m$ and defining relations $r_1=0,\ld,r_\ell=0$. The elements $u_1=x_1+J,\ld,u_m=x_m +J$ generate $R$, $r_1(u_1,\ld,u_m)=0,\ld,r_\ell=r_\ell(u_1,\ld,u_m)=0$ in $R$ and any other relation among $u_1,\ld,u_m$ is a consequence of these relations. We write $R=\langle x_1,\ld,x_m\;|\; r_1,\ld,r_\ell\rangle$. If $S$ is another algebra, generated by some elements $v_1,\ld,v_m$ so that $r_1(v_1,\ld,v_m)=0,\ld,r_\ell=r_\ell(v_1,\ld,v_m)=0$ then the mapping $u_1\mapsto v_1,\ld,u_m\mapsto v_m$ extends to a surjective homomorphism $\vp$ of $R$ onto $S$.  If $\dim S\ge\dim R$ then $\vp$ is an isomorphism. 

Additionally, suppose $A_m$ is given a grading by a group $G$ by assigning the degrees $g_1,\ld,g_m$ of $G$ to the free generators $x_1,\ld,x_m$. Assume that all \textit{relators} $r_1(x_1,\ld,x_m)$ are homogeneous with respect to this grading. Then $R=A/J$ naturally acquires a $G$-grading. If $S$ is a $G$-graded algebra generated by homogeneous $v_1,\ld,v_m$ of degrees $\alpha(g_1),\ld,\alpha(g_m)$ where $\alpha:G\to G$ is an automorphism then $R$ and $S$ are weakly isomorphic.

Let us give graded presentations of some of the fine graded division algebras above in terms of graded generators and defining relations.

\begin{equation}\label{eq28p}
M_2^{(8)}=(x_1,x_2\;|\;x_1^4=x_2^2=-1,\: x_1x_2=-x_2x_1),
\end{equation}
\begin{equation}\label{eq432p}
M_4^{(32)}=(x_1,x_2,x_3,x_4\;|\;x_1^4=\:x_k^2=-1 (k=2,3,4),\: x_\ell x_m=-x_mx_\ell (1\le \ell<m\le 4))
\end{equation}

\subsection{Tensor products of division gradings}\label{sstp}

Given groups $G_1, G_2,\ld,G_m$ and $G_k$-graded algebras $R_1, R_2,\ld, R_m$, $k=1,\ld,m$, one can endow the tensor product of algebras $R=R_1\ot R_2\ot\cdots\ot R_m$ by a $G=G_1\times G_2\times\cdots\times G_m$-grading, called the \textit{tensor product of gradings} if one sets 
$$
R_{(g_1,g_2,\ld,g_m)}=(R_1)_{g_1}\ot (R_2)_{g_2}\ot\cdots\ot (R_m)_{g_m}.
$$ 
Here $g_k\in G_i$, for all $k=1,2,\ldots,m$. 

In the case of division algebras over an algebraically closed field, the tensor product of two graded division algebras is a graded division algebra. This is no more true in our case. Indeed, $D_1\otimes D_2$, where $D_i=\R,\C,\H$ is a division algebra if and only if at least one of $D_i$ is $\R$. If $R$ is a $G$-graded division algebra, $S$ an $H$-graded division algebra, $(R\times S) _{(g,h)}=R_g\ot S_h$, for all $(g,h)\in G\times H$. Clearly, all nonzero elements in each homogeneous component are invertible if this is true for the identity component. As a result, a tensor product of two division gradings is a graded division grading only if the identity component of one of them is one-dimensional. This condition is not sufficient if we want to obtain a simple graded division algebra. We know that for this both tensor factors must be simple algebras. And even this is not sufficient, as shown by the example of $\C^{(2)}\ot \C^{(2)}$. 

Another question is the equivalence of different tensor product gradings. In the ungraded case, it is well-known (see \cite[Section 3.6]{G}) that $\H\ot\C\cong M_2(\R)\ot\C\cong M_2(\C)$ and $\H\ot\H\cong M_2(\R)\ot M_2(\R)\cong M_4(\R)$. 

These isomorphisms combine well with the gradings in the case where the grading of each factor is fine. In the first case, the isomorphism is given by $\vp:i\ot 1\mapsto B\ot i,\:j\ot 1\mapsto C\ot 1,\:k\ot 1\mapsto A\ot i$. Now $R=\H^{(4)}\ot\C^{(2)}$ is graded by $G=(\alpha)_2\times(\beta)_2)\times(\gamma)_2$ so that $\deg i\ot 1=\alpha$, $\deg j\ot 1=\beta$, $\deg 1\ot i=\gamma$. Let us choose the same group to grade $S=M_2^{(4)}\ot\C^{(2)}$ by $\deg A\ot 1=\alpha\beta\gamma$, $\deg B\ot 1=\alpha\gamma$, $\deg I\ot i=\gamma$. Then the supports of $M_2^{(4)}$, equal to $(\alpha\beta\gamma),\:(\alpha\gamma))$, and  $\C^{(2)}$, equal to $(\gamma)$, have trivial intersection, so this grading is a tensor product grading. Also, $\alpha\mapsto\alpha\beta\gamma,\:\beta\mapsto\alpha\gamma,\:\gamma\mapsto\gamma$ is an automorphism of $G$. This proves that $\vp:R\to S$ is a weak isomorphism of $\H^{(4)}\ot \C^{(2)}$ and $M_2^{(4)}\ot\C^{(2)}$. 

In a similar fashion one proves the weak isomorphism of $R=\H^{(4)}\ot\H^{(4)}$ and $S=M_2^{(4)}\ot M_2^{(4)}$. In this case, the graded equivalence $\psi:R\to S$ is given by $\psi(i\ot 1)=C\ot I$, $\vp(j\ot 1)=A\ot C$, $\psi(1\ot i)=I\ot C$, $\psi(1\ot j)=C\ot A$.

At the same time,  there is no graded equivalence between $R=\H^{(4)}\ot\H$, graded by $\zz$ and $S=M_2(\R)\ot M_2(\R)$, as graded tensor products. Indeed, if we give trivial grading to the second factor in $S$ then it is not a graded division algebra. If  we grade both factors by $\z$ then the identity component of the tensor product will be isomorphic to $\C\ot\C$, which is not a division algebra. Still, the map $\psi:\H^{(4)}\ot\H\to M_4(\R)$ transfers the structure of a graded division algebra, which is not a tensor product of graded algebras of smaller dimension. We will denote this grading by $M_4^{(4)}$. Its identity component is 4-dimensional, hence isomorphic to the quaternion algebra $\H$.

Note one more equivalence of graded tensor products: $R=M_2(\C,\ZZ_4)\ot\H^{(4)}$ and $S=M_2(\C,\ZZ_4)\ot M_2^{(4)}$. The presentation of $R$ in terms of generators and defining relations is
$$
(x_1,x_2,y_1,y_2\;|\;x_1^4=x_2^2=y_1^2=y_2^2=-1, x_1x_2=-x_2x_1, y_1y_2=-y_2y_1,x_iy_j=y_jx_i).
$$
Here $\deg x_1=\alpha$, $\deg x_2=e$, $\deg y_1=\beta$, $\deg x_2=\gamma$, where $o(\alpha)=4$, $o(\beta)=o(\gamma)=2$.

The presentation of $S$ in terms of generators and defining relations is
$$
(x_1,x_2,z_1,z_2\;|\;x_1^4=x_2^2=-z_1^2=-z_2^2=-1, x_1x_2=-x_2x_1, y_1y_2=-y_2y_1,x_iy_j=y_jx_i).
$$
Here $\deg x_1=\alpha$, $\deg x_2=e$, $\deg z_1=\beta$, $\deg z_2=\gamma$, where $o(\alpha)=4$, $o(\beta)=o(\gamma)=2$. Clearly, in the first presentation, we can perform a graded change of generators $z_1=y_1x_1^2$ and $z_2=y_2x_1^2$, which will lead to the equivalence of $R$ and $S$.

\section{Low-dimensional graded division algebras. Main Theorem}\label{sld}

In this section we list some low-dimensional fine graded division algebras and their coarsenings. We also state our main result, which is the classification up to equivalence of simple algebras which are graded division algebras.

We have already mentioned the simplest examples of real graded division algebras, which are $\R$, $\C$, $\H$, $\C^{(2)}$, $\H^{(2)}$, $\H^{(4)}$ and $M_2^{(4)}$. In terms of homogeneous Clifford gradings,  $\C^{(2)}\cong \cg(1,0)$, $\H^{(4)}\cong \cg(2,0)$ and $M_2^{(4)}\cong\cg(0,2)$. These gradings are fine. Now $\C$ is a coarsening of $\C^{(2)}$, whereas $\H$ is a coarsening of $\H^{(2)}$ and $\H^{(2)}$ is a coarsening of $\H^{(4)}$. A coarsening $R$ of $M_2^{(4)}$, which is denoted by $M_2^{(2)}$, is obtained if we set $R_e=\langle I,A\rangle$ and $R_\alpha=\langle B,C\rangle$. This is a grading by $G=(\alpha)_2$.

Since $M_2(\R)$ can be written as $C(1,1)$ and $C(0,2)$, we can make it in a fine graded division algebra in two ways: as $\cg(1,1)$ and $\cg(0,2)$. There are different ways of establishing isomorphisms between Clifford algebras and matrix algebras (see \cite[Section 6.2]{G}). In the case of  $\cg(1,1)\cong M_2(\R)$, this is given by $e_1\mapsto C$, $e_2\mapsto B$. So $R_\alpha =\langle C\rangle$, $R_\beta=\langle B\rangle$, $R_{\alpha\beta}=\langle A\rangle$. 
In the case of $\cg(0,2)\cong M_2(\R)$, by $e_1\mapsto B$, $e_2\mapsto A$. As a result,
$R_\alpha =\langle B\rangle$, $R_\beta=\langle A\rangle$, $R_{\alpha\beta}=\langle C\rangle$.

From this matrix presentation, it is clear that $R=\cg(1,1)$ and $R'=\cg(0,2)$ are weakly isomorphic, that is, there is an automorphism $\vp:\zz\to\zz$ and a linear automorphism $f: \R^2\to\R^2$ such that $fR_g f^{-1}=R'_\vp(g)$, for any $g\in\zz$. However, because we deal with \textit{metric spaces}, there is no \textit{isometry} $f: (\R^2,x_1^2-x^2_2)\to(\R^2,x_1^2+x^2_2)$. So we may say that the gradings $\cg(1,1)$ and $\cg(0,2)$ are not \emph{isometrically} weakly isomorphic.

The only 8-dimensional real simple algebra is $R=M_2(\C)$. Since $M_2(\C)\cong M_2(\R)\ot\C\cong C(2,1)$, we can obtain division gradings on $R$ in various ways. Using tensor products, we obtain a fine division grading $M_2^{(4)}\ot\C^{(2)}$, where the grading group is $\z^3\cong(\alpha)_2\times(\beta)_2\times(\gamma)_2$, and the generating components are $R_\alpha=\langle A\ot 1\rangle$, $R_\beta=\langle B\ot 1\rangle$, $R_\gamma=\langle I\ot i\rangle$. The graded presentation of this algebra is 
\begin{equation}\label{eq10}
 (x_1,x_2,x_3\;|\; x_1^2=x_2^2=-x_3^2=1, x_1x_2=-x_2x_1,x_1x_3=x_3x_1,x_2x_3=x_2x_3).
\end{equation}
  Its $\zz\cong G/(\gamma)$-factor-grading is the Pauli grading on $M_2(\C)$, that is a grading given by the same formula as in (\ref{ePauli2}) but the linear spans must be taken over $\C$ rather than over $\R$.  One can also consider $M_2^{(2)}\ot\C^{(2)}$, which is another $\zz\cong G/(\alpha\beta)$-factor-grading of $M_2^{(4)}\ot\C^{(2)}$. Now since $M_2(\C)\cong C(2,1)$, a fine division grading can be given as $\cg(2,1)$. In terms of graded generators and defining relations this is
\begin{equation}\label{eq11}
 (y_1,y_2,y_3\;|\; y_1^2=y_2^2=-y_3^2=-1, y_py_q=-y_qy_p), 1\le p<q\le 3.
\end{equation}
Clearly a graded change of variables $x_1=y_2y_3$, $x_2=y_1y_3$, $x_3=y_1y_2y_3$ transforms (\ref{eq11}) to (\ref{eq10}).

In the case $G=\z$, we have $R=M_2^{(2)}\ot \C$. This is not a division grading because then $R_e\cong \C\ot\C$, which is not a division algebra. In the case $G=\zz$, we have  $M_2^{(2)}\ot \C^{(2)}$ and the Pauli grading $M_2^{(4)}\ot \C$. In the case $\ZZ_2^3$ we have $M_2^{(4)}\ot \C^{(2)}$.

All the previous examples were graded by an elementary abelian 2-group. A more subtle fine division grading on $M_2(\C)\cong C(2,1)$ appears when the grading group $G$ involves elements of order 4. We have described this ealier when we talked about non-homogeneous Clifford gradings. We denoted this grading by $M_2^{(8)}$.

We know that $\H\ot\H\cong M_2(\R)\ot M_2(\R)$. This isomorphism transfers the structure of graded division algebra from $(\H^{(4)}\ot\H)$ to $M_4(\R)\cong M_2(\R)\ot M_2(\R)$. As we mentioned earler, this grading of $M_4(\R)$ is not a tensor product of gradings on the tensor factors $M_2(\R)$. Thus $R=M_4(\R)$ acquires a division $\zz\cong(\alpha)_2\times(\beta)_2$-grading $M_4^{(4)}$ whose components are as follows (as usual, $\gamma=\alpha\beta$):
\begin{eqnarray}\label{eq12}
R_e&=&\langle I\ot I, C\ot I,A\ot C,B\ot C\rangle,\\R_\alpha&=&(I\ot C)R_e,\;R_\beta=(C\ot A)R_e,\;R_\gamma=(C\ot B)R_e,\nonumber
\end{eqnarray}

Finally, notice that the natural isomorphism $\vp:\H\ot\C\to M_2(\R)\ot\C$ defined by $\vp(i\ot z)=A\ot iz$, $\vp(j\ot z)=B\ot iz$ induces graded isomorphisms for the refinements $\vp:H^{(2)}\ot\C^{(2)}\to
M_2^{(2)}(\R)\ot\C^{(2)}$ and $\vp:H^{(4)}\ot\C^{(2)}\to
M_2^{(4)}(\R)\ot\C^{(2)}$. 

So the full list of ``building blocks'' which will be used to desrcribe all simple graded division algebras is as follows: 
\begin{itemize}
\item $\R$, 
\item $\C$ and its refinement $\C^{(2)}\cong\cg(1,0)$, 
\item $\H$, and its refinements $\H^{(2)}$ and  $\H^{(4)}\cong\cg(2,0)$, 
\item $M_2^{(2)}$ and its refinement $M_2^{(4)}\cong\cg(1,1)\cong\cg(0,2)$, 
\item $M_2^{(2)}\ot\C^{(2)}$ and its refinement $M_2^{(8)}\cong\cg(2,1)\cong\cg(0,3)$,
\item $M_2(\C,\ZZ_4)$ and its refinement $M_2^{(8)}$, weakly isomorphic to a non-homogeneous grading on $C(2,1)$ described in  Subsection \ref{ssogc},
\item $M_4^{(4)}\cong \H^{(4)}\ot\H$ and its refinements $ \H^{(4)}\ot\H^{(2)}$ and $ \H^{(4)}\ot\H^{(4)}\cong M_2^{(4)}\ot M_2^{(4)}\cong\cg(2,2)$,
\item $M(G,\beta)$, where $\beta$ is a non-singular skew-symmetric complex bicharacter on $G=H\times H$ (Pauli grading).
\end{itemize}

Our main goal is to prove the following.
\begin{theorem}\label{tgdar} Any division grading on a real simple algebra $M_n(D)$, $D$ a real division algebra, is  weakly isomorphic to one of the following types
\begin{enumerate}
\item[$D=\R:$]
\begin{enumerate}
\item[\rm (i)] $(M_2^{(4)})^{\ot k}\cong\cg(k,k)$;
\item[\rm (ii)] $M_2^{(2)}\otimes (M_2^{(4)})^{\ot (k-1)}$,  a coarsening of {\rm (i)};
\item[\rm (iii)]$M_4^{(4)}\ot(M_2^{(4)})^{\ot (k-2)}$,  a coarsening of {\rm (i)};
\end{enumerate}
\item[$D=\H:$]
\begin{enumerate}
\item[\rm (iv)] $\H^{(4)}\ot(M_2^{(4)})^{\ot k}\cong\cg(k+1,k-1)$;
\item[\rm (v)] $\H^{(2)}\otimes (M_2^{(4)})^{\ot k}$ a coarsening of {\rm (iv)};
\item[\rm (vi)] $\H\ot(M_2^{(4)})^{\ot k}$, a coarsening of {\rm (v)};
\end{enumerate}
\item[$D=\C:$]
\begin{enumerate}
\item[\rm (vii)] $\C^{(2)}\ot(M_2^{(4)})^{\ot k}\cong\cg(k+1,k)$,
\item[\rm (viii)] $\C^{(2)}\ot M_2^{(2)}\ot(M_2^{(4)})^{\ot (k-1)}$, a coarsening of {\rm (vi)};
\item[\rm (ix)] $M_2^{(8)}\ot (M_2^{(4)})^{\ot (k-1)}$,
\item[\rm (x)] $ M_2(\C,\ZZ_4)\ot(M_2^{(4)})^{\ot (k-1)}$, a coarsening of {\rm (ix)};
\item[\rm (xi)] $(M_2^{(4)})^{\ot k}\ot\C^{(2)}\ot\H$,
\item[\rm (xii)] $M_2^{(8)}\ot (M_2^{(4)})^{\ot (k-1)}\ot\H$
\item[\rm (xiii)] Pauli grading.
\end{enumerate}
\end{enumerate}
None of the gradings of different types or of the same type but with different values of $k$ is weakly isomorphic to the other. 
\end{theorem}

\begin{remark}\label{rc}It is worth mentioning that Types (i) to (viii) are homogeneous Clifford gradings or their coarsenings while Types (ix) grading is a non-homogeneous Clifford grading and (x) is its coarsening.
\end{remark}

\begin{proof}Here we will only prove that none of the gradings on the list is weakly isomorphic to the other. In the following sections we will prove that every division grading is weakly isomorphic to one of the gradings on the list.

Clearly, there is no weak isomorphism of  $M_{n_1}(D_1)$ and $M_{n_1}(D_1)$ if $(n_1,D_1)$ and $(n_2,D_2)$ are different.  The next invariant is the dimension of the neutral component $R_e$. These invariants are sufficient to separate   algebras (i) to (vi) from other algebras on the list and from each other. As for gradings (vii) to (xiii), then Pauli gradings are the only ones where all elements of the center have degree $e$. In the remaining cases (vii) to (xii),  $\dim R_e=1$ in (vii) and (ix), $\dim R_e=2$ in cases (viii) and (x),  $\dim R_e=4$ in cases (viii) and (x). Finally, in cases (ix), (x) and (xii) the grading group has elements of order 4 while in (vii), (viii) and (xi) the grading group is an elementary abelian 2-group.
\end{proof}

\section{Two lemmas}\label{sstl}

Let $R$ be a $G$-graded division algebra over $\R$, $G$ a finite abelian group. We denote by $Z(R)$ the center of $R$. If $Z(R)=\R.I$ then $R$ is called central simple. Otherwise, $R\cong M_n(\C)$, for some $n\ge 1$, and $Z(R)=\C.I$. It is well-known that $Z(R)$ is a graded subalgebra of $R$. So $iI$ is always a homogeneous element but its degree does not need to be $e$.

The following lemmas are true.

\begin{lemma}\label{l1} 
If $\dim {R_e}=1$ then, for any homogeneous $u,v\in R$, we have $uv=\pm vu$. If also $\dim Z(R)=1$, then $G$ is an elementary abelian 2-group and one can choose $u_g$ in each $R_g$ so that $u_g^2=\pm I$.
\end{lemma}

\begin{proof} 
Indeed, for any $g\in G$ and any $u\in R_ g$, the mapping $x\mapsto uxu^{-1} $ is a graded automorphism of $R$.  It follows that for any homogeneous element $v\in R_h$ we have $uvu^{-1}\in R_h=\R v$. Hence $uvu^{-1} =\lambda v$, for a real number $\lambda$. Now since  $G$ is finite, there is natural $m$ such that $g^m=e$.  Then $u^m\in R_e$ so that $v=u^mvu^{-m} =\lambda^m v$.  As a result, $\lambda^m=1$,  hence $\lambda =\pm 1$, proving $uv=\pm vu$. It follows that $u^2\in Z(R)$, the center of $R$. We know that the center of an algebra is a graded subalgebra. Hence if $Z(R)$ is one-dimensional it must be equal to $R_e$. Thus $g^2=e$, for any $g\in G$, proving that $G$, indeed, is an elementary 2-group. Clearly then one can choose $u_g\in R_g$ so that $u^2_g=\pm I$.
\end{proof}

\begin{lemma}\label{leh} Let $\dim R_e=4$. Then $R\cong\H\ot S$, where $S$ is a graded division algebra with $\dim S_e=1$ so that $R_e=\H\ot 1$.
\end{lemma}

\begin{proof} 
Let $u_g\in R_g$. We have $R_g=R_eu_g$. Consider the map $x\mapsto u_gxu_g^{-1}$, where $x\in R_e$. This is an automorphism of $R_e$. All automorphisms of the quaternion algebra over $\R$ are inner (see \cite[Theorem 1.16]{BR}), hence there is $q_g\in R_e$ such that $u_gxu_g^{-1}=q_gxq_g^{-1}$, for all $x\in R_e$. Setting 
$v_g=q_g^{-1}u_g$ we are able to find representatives in each $R_g$ which commute with every element in $R_e$.  Now $v_gv_h\in R_{gh}$ hence $v_gv_h=\alpha(g,h)v_{gh}$, for all $g,h\in G$ and some $\alpha(g,h)\in R_e$. Since all $v_g,v_h,v_{gh}$ are in the centralizer of $R_e$, the same is true for $\alpha(g,h)$. Hence all $\alpha(g,h)$ are real numbers. As a result, the real linear span $S$ of all $v_g$ is a real subalgebra of $R$. By comparison of dimensionas, we observe that $R\cong\H\ot S$. The subalgebra $S$ is simple and graded. Each homogeneous component of $S$ is finite-dimensional so that $S$ is a graded divisional algebra with $\dim S_e=1$.
\end{proof}

\section{Central simple graded division algebras}\label{scsg}
These are $R=M_n(\R)$ and $M_n(\H)$.
\subsection{Case $\dim R_e=1$}\label{ssrr}
Let us view $G$ as a vector space over $\z$ and define a bilinear (skew)symmetric form $\beta$ on $G$ with values in $\z$ from the formula $u_gu_h=(-1)^{\beta(g,h)}u_hu_g$. This splits $G$ as an orthogonal product $G=G_1\times G_2\times\cdots G_k\times \Ker\beta$, where each $G_s$ is isomorphic to $\zz=(a_s)_2\times(b_s)$, $i=1,\ld,k$. Now if $g\in \Ker\beta$ then $u_g$ commutes with all other $u_h$. Since $R$ has only trivial center, we have $g=e$. Thus $G=G_1\times G_2\times\cdots G_k$. If $R_s=\bigoplus_{g\in G_s} R_g$, $s=1,\ld,k$ then each $R_s$ is a 4-dimensional subalgebra whose elements commute with the elements of all $G_r$ where $r\ne s$. Each $R_s$ is generated by homogeneous $e_s=u_{a_s}$ and $f_s=u_{b_s}$ such that $e_sf_s=-f_se_s$ and $e_s^2, f_s^2=\pm I$. Thus each $R_s$ is $\cg(p,m)$, with $p+m=2$. One easily checks that $\cg(2,0)\cong\H$, with the grading $H_\alpha=\langle i\rangle$, $H_\beta=\langle i\rangle$, $H_{\alpha\beta}=\langle k\rangle$.

\begin{remark}\label{rca}At the same time, there is a regular procedure (see \cite[Section 6.3]{G}) which allows one to switch from arbitrary Clifford algebra $C(p,m)$ with $p+m$ even to the tensor product of 4-dimensional algebras. During this procedure homogeneous elements are mapped to homogeneous ones. If we reverse this procedure, we may switch from the tensor product of Clifford gradings $\cg(p,m)$, where $p+m=2$ to one single fine Clifford grading $\cg(p,m)$. For different values of $p,m$, these gradings are pairvise isometrically inequivalent. However, if we do not require isometric equivalence, then for each $n=2k$ we have only two types of fine Clifford gradings $\cg(p,m)$, where $p+m=n$. These are $\cg(k,k)\cong M_{2^k}(\R)$ and $\cg(k+1,k-1)\cong M_{2^{k-1}}(\H)$ (see \cite[Corollary 6.3.1]{G}).
\end{remark}

As an illustration, consider $R=\cg(1,1)\ot\cg(0,2)$. We have that $R$ is given by presentation $(e_1,e_2,f_1,f_2\:|\:e_1^2=-1, e_2^2=f_1^2=f_2^2=1;e_1e_2=-e_2e_1, e_1f_1=f_1e_1,e_1f_2=f_2e_1,e_2f_1=f_1e_2,f_1f_2=-f_2f_1)$. Let us consider $e_3=e_1e_2f_1$, $e_4=e_1e_2f_4$. Then $e_3^2=e_4^2=1$, $e_3e_4=-e_4e_3$. Finally, $e_1e_3=e_1e_1e_2f_1=-e_2f_1$, while $e_3e_1=e_1e_2f_1e_1=e_2f_1$. Similarly, $e_1e_4=-e_4e_1$, 
$e_2e_3=-e_3e_2$ and $e_2e_4=-e_4e_2$. $R$ is now given by presentation  $(e_1,e_2,e_3,e_4\:|\:e_1^2=-1, e_2^2=e_3^2=e_4^2=1;e_1e_2=-e_2e_1, e_1e_3=-e_3e_1,e_1e_4=e_4e_1,e_2e_3=e_3e_2,e_3e_4=-e_4e_3)$.If $\deg e_1=\alpha$, $\deg e_2=\beta$, $\deg f_1=\gamma$, $\deg f_2=\delta$, with all group elements linearly independent,  then $\deg e_1=\alpha$, $\deg e_2=\beta$, $\deg e_3=\alpha\beta\gamma$, $\deg e_4=\alpha\beta\delta$, all group elements linearly independent. So we obtain $R=\cg(1,3)$. As an algebra, $R$ is isomorphic to $M_4(\R)$.

In the conclusion of this subsection, we have the following.
\begin{proposition}\label{pssrr} Any fine division grading on a central simple algebra is equivalent to either $\cg(k,k)\cong M_{2^k}(\R)$ or $\cg(k+1,k-1)\cong M_{2^{k-1}}(\H)$. These algebras also can be written as tensor product gradings $(M_2^{(4)})^{\ot k}$ or $(M_2^{(4)})^{\ot (k-1)}\ot\H^{(4)}$. Their graded presentations in terms of generators and defining relations are 
$$
(x_1,\ld,x_{2k}\;|\;x_p^2=-1,x_q^2=1, x_ux_v=-x_vx_u)
$$
where $1\le p\le k$, $k+1\le q\le 2k$, $1\le u< v\le 2k$,
and
$$
(x_1,\ld,x_{2k};|\;x_p^2=-1, x_q^2=-1, x_ux_v=-x_vx_u),
$$
where $1\le p\le k+1$, $k+2\le q\le 2k$, $1\le u< v\le 2k$. In both presentations, the grading group $G$ is an elementary abelian 2-group, $\deg x_s=\alpha_s\in G$ and $\{\alpha_1,\ld,\alpha_{2k}\}$ is a basis of $G$.
\end{proposition}

\subsection{Case $\dim R_e=4$}\label{ss4} If $\dim R_e=4$ then by Lemma \ref{leh}, $R\cong\H\ot S$ where $S$ is a central simple graded division algebra with $\dim S_e=1$. Thus we need only to deal with the case $\dim R_e=1$. By Lemma \ref{l1}, in this case $G$ is an elementary abelian 2-group, $R_g=\R.u_g$, for each $g\in G$, $u_gu_h=\pm u_hu_g$, for all $g\neq h$ and $u_g^2=\pm I$. We can apply Proposition \ref{pssrr} to $S$. We obtain either $S\cong (M_2^{(4)})^{\ot k}$ or $S\cong (M_2^{(4)})^{\ot (k-1)}\ot \H^{(4)}$. Then $R\cong (M_2^{(4)})^{\ot k}\ot\H$ or $R\cong (M_2^{(4)})^{\ot k}\ot (\H^{(4)}\ot\H)\cong(M_2^{(4)})^{\ot k}\ot M_4^{(4)}$ . 

\subsection{Case $\dim R_e=2$}\label{ss2}
Let $R=M_k(D)$, where $D\cong\R$ or $D\cong\H$. We have that $Q=R_e=\langle I, J\rangle\cong\C$, where $J^2=-I$ is noncentral. We set $Z^+=\{ x\in R\;\vert\; xJ=Jx\}$ and $Z^-=\{ x\in R\;\vert\; xJ=-Jx\}$. These are graded subspaces. We also have $xZ^+=Z^-$ and $xZ^-=Z^+$, for any $x\in Z^-$ and $R=Z^+\oplus Z^-$. Let $H=\supp Z^+$. Then $H$ is a subgroup of $G$ and $[G:H]=2$. Let $x\in Z^-$ and $y\in Z^+$ be homogeneous elements. Then $xyx^{-1}=\lambda y$, $\lambda\in Q$. The conjugation by $x$ on $Q$ is a nontrivial automorphism of $Q\cong\C$, so this is usual complex conjugation, that is,  $x\lambda=\overline{\lambda}x$. It follows that $x^2yx^{-2}=\lambda\overline{\lambda}y=|\lambda|^2y$. As before, there is $m$ such that $x^{2m}yx^{-2m}=y$. In this case, $|\lambda|^{2m}=1$ so that $|\lambda|=1$ and $x^2yx^{-2}=y$. Therefore, $x^2$ commutes with $Z^+$. Since $Z^-=xZ^+$, it follows that $x$ commutes with $Z^-$ hence with the whole of $Z$. Therefore, $x^2=\alpha I$, where $\alpha\in \R$.
As a result, we must conclude that $g^2=e$. Since this is true for any $g\in G\setminus H$, the group $G$ is an elementary abelian 2-group.

Next, if there is a homogeneous $x\in Z^-$ such that $x^2=I$ then $(Jx)^2=I$ and then
$$
A=\langle  I,J,x,Jx\rangle \cong  M_2(\R)\mbox{ with }  M_2^{(2)}\mbox{ grading  }
$$
If there is $x\in Z^-$ such that $x^2=-I$ then $(Jx)^2=-I$ and
$$
A=\langle  I,J,x,Jx\rangle \cong  \H\mbox{ with }  \H^{(2)}\mbox{ grading  }
$$
In any case, $R\cong A\otimes B$, where $B$ is the centralizer of $A$ in $R$ (since $R$ and $A$ are central simple) and $B$ is the tensor product of algebras $M_2^{(4)}$ or $\H^{(4)}$. As before, we can assume that no more than one factor is $\H^{(4)}$. Now the fllowing is true

\begin{lemma}\label{lhm} $A=M_2^{(2)}\ot\H^{(4)}$ is weakly isomorphic to $B=\H^{(2)}\ot M_2^{(4)}$.
\end{lemma}

\begin{proof} We can write the graded presentations for both algebras. For $A$ we have
\begin{eqnarray*}
A=(x_1,x_2,y_1,y_2\;|\; x_1^2=-1,x_2^2=1,y_1^2=y_2^2=-1,\\x_1x_2=-x_2x_1,y_1y_2=-y_2y_1,x_ky_\ell=y_\ell x_k,1\le k,\ell\le 2)
\end{eqnarray*}
For $B$ we have
\begin{eqnarray*}
B=(x_1,x_2,y_1,y_2\;|\; x_1^2=x_2^2=y_1^2=-1,y_2^2=1,\\x_1x_2=-x_2x_1,y_1y_2=-y_2y_1,x_ky_\ell=y_\ell x_k,1\le k,\ell\le 2)
\end{eqnarray*}
In both presentations, $\deg x_1=e$, $\deg x_2=\alpha$, $\deg y_1=\beta$, $\deg y_2=\gamma$. 

Let us replace in both presentations $y_1$ by $x_1x_2y_1$ and $y_2$ by $x_1x_2y_2$. This is a standard transformation applied when we want to switch from the tensor product of two Clifford algebras to one Clifford algebra. The elements obtained have degrees $\alpha\beta$ and $\alpha\gamma$, respectively. We denote these elements in the both presentations by $x_3,x_4$. 
One easily checks that the new graded presentations are as follows.
For $A$ we have
\begin{eqnarray*}
A=(x_1,x_2,x_3,x_4\;|\; x_1^2=-1,x_2^2=1,x_3^2=x_4^2=-1,x_kx_\ell=-x_\ell x_k, 1\le k<\ell\le 4)
\end{eqnarray*}
For $B$ we have
\begin{eqnarray*}
B=(x_1,x_2,y_1,y_2\;|\; x_1^2=x_2^2=-1,x_3^2=1,x_4^2=-1,x_kx_\ell=-x_\ell x_k, 1\le k<\ell\le 4)
\end{eqnarray*}
If we map $x_1 \to x_1$, $x_2 \to x_3$, $x_3 \to x_2$ and $x_4 \to x_4$ then $A$ is mapped isomorphically to $B$. To have weak isomorphism we must make sure that the degrees of the above elements are mapped by a group isomorphism. We have $e\to e$, $\alpha\to\alpha\beta$, $\alpha\beta\to \alpha$ and $\alpha\gamma\to \alpha\gamma$. Such mapping can be achieved by an automorphism mapping $\alpha\to\alpha\beta$, $\beta\to\beta$ and $\gamma\to\beta\gamma$.
\end{proof}

As a result, we have the following
\begin{proposition}\label{pcgda2} A division grading on a central simple $R=M_n(D)$ with 2-dimensional homogeneous components exists if and only if $n=2^k$. Up to weak isomorphism, such a grading is weakly equivalent to one of the following types
\begin{enumerate}
\item[$\mathrm{(i)}$] $M_2^{(2)}\otimes (M_2^{(4)})^{\ot k}$,
\item[$\mathrm{(ii)}$] $M_2^{(2)}\otimes (M_2^{(4)})^{\ot k}\otimes \H^{(4)}$.
\end{enumerate}
for some natural $k\ge 0$. All of them appear as coarsening of the gradings with one-dimensional components as in Proposition \ref{pssrr}.
\end{proposition}
\section{Gradings on non-central simple algebras}\label{sc}
In this section, we deal with the case $R=M_n(\C)$.
\subsection{Case $\dim R_e=1$}
In this case the center $Z(R)$ is 2-dimensional and isomorphic to $\C$. This is the case where $R=M_n(\C)$ and $R_e=\R I$. It follows that there is an element $a\in G$ of order 2 such that $Z(R)=(\R I)\oplus (\R iI)$ where the degree of $iI$ is equal to $a$. Set $H_1=(a)$. Now consider any homogeneous $u\in R_g$, for some $g\in G$. Then as in the proof of Lemma \ref{l1} $u^2\in Z(R)$. Since $u^2$ is homogeneous, we either have $u^2=\alpha I\in \R$ or $u^2=\alpha iI$, again with $\alpha\in \R$. Clearly, $u^4\in\R I$ and hence $g^4=e$.

If all $u^2$ are in $\R I$ then $g^2=e$, for all $g\in G$, meaning that $G$ is an elementary abelian 2-group. We can then write $G=H_1\times  H_2$, for some subgroup $H_2$ of $G$. Consider $B=\bigoplus_{h\in H_2}R_h$. Clearly, this is a graded division algebra, with $B_e=\R I$ and $Z(B)=\R I$. In this case, we can apply the argument of the previous cases and conclude that $R\cong\C^{(2)}\otimes B$ where $B$ is the tensor product of 4-dimensional simple graded division algebras, with 1-dimensional homogeneous components. As usual, we must take into account that $\H^{(4)}\otimes\H^{(4)}$ is equivalent to $ M_2^{(4)}\otimes M_2^{(4)}$. 

Now assume there is $v\in R_g$ such that $v^2\not\in\R I$ and $v^4\in\R I$. Then $g$ has order 4. If also $u\in R_h$ is such that $u^2\not\in\R I$ then, as in Lemma \ref{l1}, $uv=\pm vu$ hence $(uv)^2\in\R I$. This allows us to find an elementary abelian 2-subgroup $H$ so that $G=(g)_4\times H$, and for any homogeneous $u\in R_h$, $h\in H$, we have $u^2=\lambda I$, where $\lambda\in\R$. 

Our next claim is that there is $u\in R_h$, $h\in H$ such that $uv=-vu$. Since $v^4=-I$, the converse would mean that the centre of $R$ contains four (!) linearly independent elements $I,v,v^2$ and $v^3$.  Thus, we may choose $u\in R_h, h\in H$ such that $vu=-uv$. Also, we may choose $u$ so that $u^2=-1$. If instead,  $u^2=1$, replace $u\in R_h$ by $iu\in R_{g^2h}$. Let us set $H_1=(g)_4\times(h)_2$. Then the subalgebra $A$ generated by $u,v$ is a graded division algebra with support $H_1$.Its graded presentation is $(v,u\;|\;v^4=u^2=-1,\:uv=-vu)$. Comparing with (\ref{eq28}), we conclude that $A\cong M_2^{(8)}$. 
 
Since 4 is the highest order of elements in $G$ we can find an elementary abelian 2-subgroup $H_2$ such that $G=H_1\times H_2$ and $vw=wv$, for all $w\in R_t$, where $t\in H_2$. 

Let us consider two cases. 

In the first case, we have that $uw=wu$, for all $w\in R_t$, $t\in H_2$. Setting $B=\bigoplus_{t\in H_2}R_t$ we get $R=A\otimes B$, where $B$ and is the direct product of 4-dimensional graded division algebras.

In the second case, there is unique $t\in H_2$ and $w\in R_t$ such that $vw=-wv$. Let us write $H_2=(t)_2\times H_2^\prime$. Suppose $w$ commutes with all $R_s$, $s\in H_2^\prime$. Consider the following subalgebras: $A^\prime$ generated by $v,u,w$ and $B^\prime=\bigoplus_{s\in H_2^\prime}R_s$. Then $R=A^\prime\otimes B^\prime$. It follows that $A^\prime$ is a 16-dimensional simple algebra over $\R$ whose centre contains $I, x^2$. Such algebras do not exist! As a result, we must assume that there is unique $z\in R_s$, $s\in H_2^\prime$ such that $wz=-zw$. Set $H_0=(g)_4\times(h)_2\times (t)_2\times(s)_2$ and write $G=H_0\times T$, where $T$ is an elementary abelian group. Let $C$ be the sum of all components of $R$ labelled by $H_0$ and $D$ by $T$. In this case, $\dim D_e =1$ and so $D$ is the tensor products of 4-dimensional simple graded division subalgebras. 

The algebra $C$ is a 32-dimensional graded division algebra. Normalizing $w, z$ so that $w^2=I$ and $z^2=I$, we obtain that $C$ is an algebra with graded presentation
\begin{eqnarray}\label{eq432c} 
(v,u,w,z\;&|&\;v^4=u^2=-w^2=-z^2=-I,\\
vu=-uv,\quad vw=wv&,& vz=zv,\quad uw=-wu,\quad uz=zu,\quad wz=-zw).\nonumber
\end{eqnarray}

If we replace $u$ by $\bu=uz$ then $\bu^2=uzuz=u^2z^2=-1$, $v\bu=vuz=-uvz=-uzv=-\bu z$, $\bu w=uzw=-uwz=wuz=w\bz$ and $\bu z=uzz=zuz=z\bu$. Thus $C$ admits presentation
\begin{eqnarray}\label{eq432d} 
(v,\bu,w,z\;&|&\;v^4=\bu^2=-1,w^2=z^2=I,\\
v\bu=-\bu v,\quad vw=wv&,& vz=zv,\quad \bu w=w\bu,\quad \bu z=z\bu,\quad wz=-zw),\nonumber
\end{eqnarray}
which is the presentation of the algebra $M_2^{(8)}\ot M_2^{(4)}$.

Put together, we have the following
\begin{proposition}\label{pc1} A division grading on $M_n(\C)$ with one - dimensional components is possible only if $n=2^k$. It can have one of two forms
\begin{enumerate}
\item[$\mathrm{(i)}$] $(M_2^{(4)})^{\ot k}\ot\C^{(2)}$,
\item[$\mathrm{(ii)}$] $M_2^{(8)}\ot (M_2^{(4)})^{\ot (k-1)}$.
\end{enumerate}
\end{proposition}

\subsection{Case $\dim R_e=4$}\label{shc} We have $R_e\cong\H$ and $R\cong M_n(\C)$. Again, as in Subsection \ref{ssrr}, we apply Lemma \ref{leh} to obtain that $R= R_e\ot C(R_e)$, where $n=2^k$. Here $C(R_e)$ is the centralizer of $R_e $. Now  since $S=C(R_e)\cong M_k(\C)$ and the homogeneous components of $S$ are one-dimensional, we may invoke Proposition \ref{pc1} to get the following.

\begin{proposition}\label{pc4} A  division grading on $M_n(\C)$ with four-dimensional components are exists if and only if $n=2^k$. It is weakly isomorphic to one of the following
\begin{enumerate}
\item[$\mathrm{(i)}$] $(M_2^{(4)})^{\ot k}\ot\C^{(2)}\ot\H$,
\item[$\mathrm{(ii)}$] $M_2^{(8)}\ot (M_2^{(4)})^{\ot (k-1)}\ot\H$
\end{enumerate}
\end{proposition}

\subsection{Case $\dim R_e=2$}
In this case $ R=M_n ( \C ) $, for some natural $ n $. We also have $ R_e \cong \C$. There are two cases. In the first case, $ R_e $ is central, hence $ R_e=\C I$, where $ I $ is the identity matrix. Since $ R_g= R_e u_g=\C u_g $, for some matrix $u_g$, $ g\in G $, it follows that each $ R_g $ is a complex subspace of $ R=M_n ( \C ) $ so that our grading is a grading of $ R=M_n ( \C ) $, as a matrix algebra over the field of complex numbers.  Clearly, this is a division grading.  All such gradings have been described using generalized Pauli matrices.

In the second case, the centre $\C I$ is the direct sum of $\R I$ of identity degree $e$ and $\R iI$ of degree $\alpha $, where $\alpha^2=e$. Also, $\R_e$ is the linear span over $\R$ of the identity matrix $I$ and a non-central matrix $J$, such that $ J^2=-I$. Given $\lambda\in R_e$ such that $\lambda=\alpha I+\beta J$, $\alpha,\beta\in\R$, we set $\overline{\lambda}=\alpha I-\beta J$. Since $R_e\cong\C$, any nontrivial automorphism of $R_e$ maps $\lambda$ to $\overline{\lambda}$.


The conjugation $x\mapsto JxJ^{-1}$ is a linear transformation of order 2 on $R$. It is easy to see that in this case, there is a homogeneous $x\in R$ such that $JxJ^{-1}=-x$. Actually, $R=R^+\oplus R^-$, where
\begin{equation}\label{epm}
R^\pm=\{ z\in R\,\vert\, JzJ^{-1}=\pm z\}.
\end{equation}
Each $R^\pm$ is a graded subspace and the support $H$ of $R^+$ is a subgroup of index 2 in $G$.

If $a\in R^+$ and $b\in R^-$ are two homogeneous elements then $bab^{-1}=\lambda a$, where $\lambda\in R_e$. The conjugation by $b$ is an nontrivial isomorphism of $R_e\cong\C$, hence $b\lambda=\overline{\lambda}b$. As a result, $b^2ab^{-2}=\overline{\lambda}\cdot\lambda a=|\lambda|^2a$. 

Since $b^2\in R^+$, it follows that there is natural $m$ such that $b^{2m}\in  R_e\subset  Z(R^+)$. Then $(|\lambda|^2)^m=1$ and $|\lambda|=1$. Hence $b^2a=ab^2$, for all $a\in R^+$. Since $R^-=bR^+$, it follows that $b^2$ commutes also with $R^-$. Finally, $b^2$ is a central homogeneous element of $R$. But then $b^2=zI$, for some $z\in \C$. Now $\C I$ is a graded $\R$-subalgebra hence splits as the sum of homogeneous components $\langle I\rangle$ and $\langle iI\rangle$. Therefore, either $b^2\in\langle I\rangle$ or $b^2\in\langle iI\rangle$.

If $b^2\in\langle I\rangle$, for all homogeneous $b\in R^-$, then all elements $g\in G\setminus H$ are elements of order 2. \textit{Clearly, in this case $G$ is an elementary abelian 2-group}. Let $H_0$ be the support of $Z(R)=\C I$. Then one can write $G=H_0\times H_1$, for some subgroup $H_1$. In this case, $R'=\bigoplus_{h\in H_1} R_h$ is a graded subalgebra isomorphic to $M_n(\R)$, with $\dim R'_e=2$. We already know from Proposition \ref{pcgda2} that in this case we have one of three cases $R'\cong M_2^{(2)}\ot(M_2^{(4)})^{\ot k}$ or $R'\cong M_2^{(2)}\ot(M_2^{(4)})^{\ot k}\ot\H^{(4)}$ or else $R'\cong\H^{(2)}\ot (M_2^{(4)})^{\ot k}$ . So in this case we have 
\begin{eqnarray}\label{ec1}
R&\cong& M_2^{(2)}\ot(M_2^{(4)})^{\ot k}\ot\C^{(2)};\; R\cong M_2^{(2)}\ot(M_2^{(4)})^{\ot k}\ot\H^{(4)}\ot\C^{(2)};\\R&\cong& \H^{(2)}\ot (M_2^{(4)})^{\ot k}\ot\C^{(2)}.\nonumber
\end{eqnarray}

However, $\H^{(4)}\ot\C^{(2)}\cong M_2^{(4)}\ot\C^{(2)}$ and $\H^{(2)}\ot\C^{(2)}\cong M_2^{(2)}\ot\C^{(2)}$. Therefore, in the first case, we have only 
\begin{equation}\label{ec1i}
R\cong M_2^{(2)}\ot(M_2^{(4)})^{\ot k}\ot\C^{(2)}
\end{equation}
In the second case, there is a homogeneous $b\in R^-$ such that $b^2\not\in R_e$. Therefore, $G$ is not an elementary 2-abelian group. Still, we can prove that in this case $G\cong \ZZ_4\times\z\times\cdots\times\z$. Indeed, as shown earlier, for any $g\not\in H$ we either have $g^2=e$ or $g^2=\alpha$, where $\alpha$ is the degree of $iI$. Let us consider $\bG=G/(\alpha)$. In this group there is a subgroup $\overline{H}\cong H/(\alpha)$ such that for any $\bg\not\in\overline{H}$ we have $\bg^2=\be$. As previously, $\bG$ is an elementary abelian 2-group. This can only happen if $G$ is as claimed.

Recall that $R_e=\langle I,J\rangle$ where $J^2=-I$, $J$ is not a real multiple of $iI$. We have $G=(g)_4\times H$ where $H$ is an elementary abelian 2-group. We have an element $x\in R_g$ such that $x^2=iI\in Z(R)$ and $xJ=-Jx$. Let us consider a subalgebra $A$ in $R$ with support $(g)_4$. Then the dimension of $A$ is 8 and it is generated by $J,x$ with relations $J^2=-I$, $x^4=-I$ and $xJ=-Jx$. This grading has appeared in Subsection \ref{ss2} and it was denoted by $M_2(\C,\ZZ_4)$. 

In order to proceed, let us denote by $D$ the subspace in $R$ of the form $D=\bigoplus_{h\in H}R_h$. Every component of this subspace is two-dimensional and one can write $R_h=R_h^+\oplus R_h^-$ where each $R_h^+$ is spanned by a vector $u_h$ which commutes with $x$, as in the previous paragraph, and $R_h^-$ is spanned by $v_h$, which anticommutes with $x$. This follows because $JR_h=R_h$ and $xJ=-Jx$. Indeed, since $wx=\pm xw$, for all homogeneous $w$, if $wx=xw$, for $w\in R_h$ then $(Jw)x=(Jx)w=(-xJ)w=-x(Jw)$. Similarly, if $wx=-xw$, for $w\in R_h$ then $(Jw)x=x(Jw)$.

Let us consider $B=\bigoplus_{h\in H}B_h$, where $B_h=R_h^+$. This is a graded subalgebra all of whose homogeneous components one-dimensional, hence a graded division algebra of the one of the types $ (M_2^{(4)})^{\ot m}$ or  $(M_2^{(4)})^{\ot m}\ot\H^{(4)}$. As vector spaces, $R\cong A\otimes B$. If $A$ and $B$ commute (actually, we only need $A$ and $J$ commute) then $R\cong A\otimes B$ as graded algebras. So in this case we have algebras of the form $M_2(\C,\ZZ_4)\ot  (M_2^{(4)})^{\ot m}$ or  $M_2(\C,\ZZ_4)\ot (M_2^{(4)})^{\ot m}\ot\H^{(4)}$. 

Now since $uJ=\pm Ju$, for all homogeneous $u\in B$, we have a linear form $\vp$ on $H$, viewed as a linear space over $\z$ defined by $uJ=(-1)^{\vp(h)}Ju$. Since all homogeneous elements of $B$ either commute or anticommute, there is a bilinear (skew)symmetric form $\sigma$ on $H$ with values in $\z$ defined by $u'u''=(-1)^{\sigma(h',h'')}u''u'$ for $u'\in B_{h'}$ and $u''\in B_{h''}$. One easily checks the linearity of $\vp$ and bilinearity of $\sigma$. Let $H'=\Ker\vp$. In our present case, $H'\neq H$, that is, there is $y\in B_k$ such that $yJ=-Jy$. Let $B'=\bigoplus_{h\in H'}$. Then $B'$ is a central graded division algebra. Let $M$ be subalgebra generated by $J,x,y$. The elements of $B'$ commute with $J$ and $x$. If also all elements of $B'$ commute with $y$ then we have that $R=M\ot B'$. In this case, $M$ is simple algebra with 2-dimensional center and of dimension 16. Such algebras do not exist. As a result, there is $z\in B_\ell$ such that $zy=-yz$. 

Let $H_1=(k)\times(\ell)$. The restriction of $\sigma$ to $H_1$ is nonsingular, and so $H=H_1\times H_2$, where $H_2$ is an orthogonal complement to $H_1$ with respect to $\sigma$. Let $N$ be a subalgebra generated by $J,x,y,z$ and $C$ the subalgebra of $R$ with support $H_2$. Then $R=N\ot C$. As before, $C$ is a simple graded division algebra with 1-dimensional homogeneous components. It follows that $C\cong (M_2^{(4)})^{\ot m}$, or $(M_2^{(4)})^{\ot m}\ot\H^{(4)}$, for some integral $k$.

As for $N$, we have that $y^2$ and $z^2$ are in $B_e\cong \R$. Replacing them, if necessary, by $iy$ and $iz$, we may assume $y^2=-I$ and $z^2=I$. Then consider the subalgebra $N=\alg\{J,x,y,z\}$. This is a 32-dimensional real graded algebra with graded presentation 
\begin{eqnarray}\label{eq432c1}
(x,\; J,\; y,\; z\;|\;J^2=-I,&\;&x^4=-I,y^2=-I,\;z^2=I\\
xJ=-Jx,\; xy=yx,\; xz=zx,&\; &Jy=-yJ,\;  Jz=zJ,\;  yz=-zy)
\end{eqnarray}
A quick comparison with (\ref{eq432c}) shows that we have obtained a coarsening of $M_4^{(32)}$, denoted earlier by $M_4(\C,\ZZ_4\times\z^2)$. As we noted earlier, $M_4^{(32)}$ is weakly isomorphic to $M_2^{(8)}\ot M_2^{(4)}$ by means of  replacement of $J$ by $Jz$. One can also go from $M_2^{(8)}\ot M_2^{(4)}$ by means of inverse operation. As a result, if, in the presentation of $M_2^{(8)}\ot M_2^{(4)}$ we set $\deg x=\alpha$, $\deg J=\gamma$, $\deg y=\beta$, $\deg z=\gamma$ then after replacement of $J$ by $Jz$, we will get exactly our graded algebra $N$. As a result, we have obtained an algebra weakly isomorphic to $M_2(\C,\ZZ_4)\ot M_2^{(4)}$.

\begin{proposition}\label{pc2}
Any graded division grading on $R=M_n(\C)$, whose graded components are 2-dimensional, belongs to one of the types:
\begin{enumerate}
\item[$\mathrm{(i)}$] Pauli grading,
\item[$\mathrm{(ii)}$] $M_2^{(2)}\ot(M_2^{(4)})^{\ot k}\ot\C^{(2)}$
\item[$\mathrm{(iii)}$] $ M_2(\C,\ZZ_4)\ot(M_2^{(4)})^{\ot k}$,
\end{enumerate}
\end{proposition}

\section{Structure of finite-dimensional simple real algebras}

As is well-known (see for example \cite[Chapter 2]{EK},\cite{NVO} or other sources) that for any Artinian $G$-graded simple algebra $R$, $G$ an abelian group, there exist a graded division algebra $D$ and a graded left-$R$, right-$D$ module $U$ such that $R$ is isomorphic to  $\End_DU$ as graded algebras. The grading on $\End_DU$ is canonical in the following sense. Let $H$ be the support of $D$, which is a subgroup in $G$. Then $U$ is the sum of irreducible $D$-modules (1-dimensional right vector subspaces over $D$) which are graded isomorphic to the shifts $D^{[g]}$. If $g_1H=g_2H$ then 
$D^{[g_1]}\cong D^{[g_2 ]}$, as graded $D$-vector spaces. Now given $A\in G/H$, we denote by $U(A)$ the sum of all $D^{[g]}$ in $U$ such that $gH=A$. Then we have a function $\kappa:G/H\to\NN\cup\{ 0\}$ given by $\kappa(A)=\dim_DU(A)$, for any $A\in G/H$. We call $\kappa$ the dimension function for $R$. It follows that with each simple graded Artinian algebra one can associate a pair $(D,\kappa)$, where $D$ is a $G$-graded division algebra with support $H$ and $\kappa$ is a function $\kappa: G/H\to\NN_0$. Note that $D$ must be simple as a non-graded algebra. Conversely, using $(D,\kappa)$, as just above, one can construct a simple graded algebra $\End_DU$. Let us denote this algebra by $\cM(D,\kappa)$. 

\begin{definition} A pair $(D,\kappa)$ is called equivalent to $(D',\kappa'))$ (we write $(D,\kappa)\sim(D',\kappa')$) if $D\cong D'$, with the same support $H$, and there are $g\in G$ and a permutation $\pi:G/H\to G/H$ such that $\kappa(A)=\kappa'(\pi(A))$ and $A=g\pi(A)$, for all $A\in G/H$.
\end{definition}

Then we have the following (cf \cite[Corollary 2.1.2]{EK}).

\begin{proposition}\label{pEK}
Any simple graded Artinian associative algebra is isomorphic to $\cM(D,\kappa)$, for some simple graded division algebra $D$ with support $H$ and a dimension function $\kappa$. Two algebras $\cM(D,\kappa)$ and $\cM(D',\kappa')$ are isomorphic if and only if $(D,\kappa)\sim (D',\kappa')$.
\end{proposition}

As a result, our main theorem sounds as follows.
\begin{theorem}\label{tM}
Any real simple finite-dimensional algebra endowed with a grading by a finite abelian group is isomorphic to an algebra $\cM(D,\kappa)$, where $D$ is weakly isomorphic to one of the algebras in the list of Theorem \ref{tgdar}. Two algebras $\cM(D,\kappa)$ and $\cM(D',\kappa')$ are isomorphic if and only if $(D,\kappa)\sim (D',\kappa')$.
\end{theorem}

Similar is the situations with locally finite simple finitary real algebras. Note that their  gradings, up to the classification of graded division algebras, have been classified in \cite{BBZ}, which generlized results of an earlier paper \cite{BZ10}. The result sounds like Theorem \ref{tM} but now $G$ can be any infinite abelian group, $H$ a finite subgroup and for each $A\in G/H$, $\kappa(A)$ is a (finite or infinite) cardinal.


\begin{thebibliography}{999}

\bibitem{BSZ}  Bahturin,Y.; Sehgal, S; Zaicev, M. {\it Group gradings on associative algebras}. J. Algebra {\bf 321} (2002), no. 1, 264--283.

\bibitem{BBZ}  Bahturin,Y.; Bre\v{s}ar, M.; Kochetov, M. {\it Group gradings on finitary simple Lie algebras}. Int. J. Algebra and Appl. {\bf 321} (2011). 

\bibitem{BK}
Bahturin, Y.; Kochetov, M. {\it Classification of group gradings on simple Lie algebras of types A, B, C and D}. J. Algebra \textbf{324} (2010), 2971--2989.

\bibitem{BShZ}Bahturin, Y.; Shestakov, I.; Zaicev, M. {\it Gradings on simple Jordan and Lie algebras}. J. Algebra, {\bf 283} (2005), no. 2, 849--868.

\bibitem{BZ02} Bahturin, Y.; Zaicev, M., \textit{Group gradings on matrix algebras}, Canad. Math. Bull., {\bf 45}(2002), 499 - 508.

\bibitem{BZ03}
Bahturin, Y.; Zaicev, M. {\it Graded algebras and graded identities}. Polynomial identities and combinatorial methods (pantelleria, 2001) Lect. Notes in Pure and Appl. Math. {\bf 235} Dekker, new York, 2003, 101  - 139.

\bibitem{BZ07}
Bahturin, Y.; Zaicev, M. {\it Involutions on graded matrix algebras}. J. Algebra, {\bf 315} (2007), no. 2, 527--540.

\bibitem{BZ10}
Bahturin, Y.; Zaicev, M. {\it Gradings on simple algebras of finitary matrices}. J. Algebra, \textbf{324} (2010), no. 6, 1279-1289.

\bibitem{EK} Elduque, A.; Kochetov, M. {\it Group gradings on simple Lie algebras}. AMS Math. Surv. Monographs  \textbf{189} (2013), 336p.

\bibitem{G} Garding, D.J.H. {\it Clifford Algebras: An Introduction}, LMS Student Texts  \textbf{78}(2011), 200p.

\bibitem{JSR} Jacobson, N. {\it Structure of Rings}. Colloquium Publications, {\bf 37}, American Math. Society, Providence, RI, 1964.

\bibitem{NVO82}
Nastasescu, C.; Van Oystaeyen, F. {\it Graded ring theory}. North-Holland Mathematical Library, {\bf 28}, North-Holland Publishing Co., 1982.

\bibitem{NVO}
Nastasescu, C.; Van Oystaeyen, F. {\it Methods of Graded Rings}. Lecture Notes in Mathematics, {\bf 1836}, Springer Verlag, 2004.

\bibitem{BR} Rosenfeld, B. {Geometry of Lie Groups.}  Mathematics and its Applications, 393. Kluwer Academic Publishers Group, Dordrecht, \textbf{1997}, xviii+393 pp.
\end{thebibliography}
\end{document}